\UseRawInputEncoding
\documentclass[12pt]{article}
\usepackage[centertags]{amsmath}
\usepackage{amsfonts}
\usepackage{amssymb}
\usepackage{latexsym}
\usepackage{amsthm}
\usepackage{newlfont}
\usepackage{graphicx}
\usepackage{listings}
\usepackage{booktabs}
\usepackage{abstract}
\date{}

\newlength{\defbaselineskip}
\newcommand{\setlinespacing}[1]%
           {\setlength{\baselineskip}{#1 \defbaselineskip}}

\newcommand{\N}{{\mathbb{N}}}

\newcommand{\actaqed}{\hfill $\actabox$}
{\medskip\noindent \textit{Proof of #1. }}%
{\actaqed \medskip}

\def\cA{{\mathcal A}}
\def\C{{\mathcal C}}
\def\cC{{\mathcal C}}

\def \Tr{\mathcal T}

\def \cK{\mathcal K}

\def \cV{\mathcal V}

\def \E{\mathcal E}
\def \cE{\mathcal E}
\def \bbE{\mathbb E}

\def\R{{\mathbb R}}
\def\Z{\mathbb Z}

\def \T{\mathbb T}

\def\bE{\mathbb E}

\def \<{\langle}
\def\>{\rangle}

\def \La{\Lambda}

\def \e{\varepsilon}

\def\la{\lambda}

\def\ba{\mathbf a}
\def\bx{\mathbf x}
\def\by{\mathbf y}
\def\bz{\mathbf z}

\def\bu{\mathbf u}
\def\bv{\mathbf v}
\def\bw{\mathbf w}

\def\bp{\mathbf p}
\def\bs{\mathbf s}
\def\btt{\mathbf t}

\def\bW{\mathbf W}
\def\bH{\mathbf H}

\def\bE{\mathbf E}

\newtheorem{Theorem}{Theorem}[section]
\newtheorem{Lemma}{Lemma}[section]

\newtheorem{Proposition}{Proposition}[section]

\newtheorem{Corollary}{Corollary}[section]
\numberwithin{equation}{section}

\newcommand{\be}{\begin{equation}}
\newcommand{\ee}{\end{equation}}

\begin{document}

\title{ Sampling discretization error of integral norms for function classes with small smoothness}
\author{  V.N. Temlyakov\thanks{University of South Carolina, Steklov Institute of Mathematics, and Lomonosov Moscow State University.
Email:  temlyakovv@gmail.com }}
\maketitle
\begin{abstract}
{We consider infinitely dimensional classes of functions and instead of the relative error setting, which was used in previous papers on the integral norm discretization, we consider the absolute error setting. We demonstrate how known results from two areas of research -- supervised learning theory and numerical integration -- can be used in sampling discretization of the square norm on different function classes. We prove a general result, which shows that the sequence  of entropy numbers of a function class in the uniform norm dominates, in a certain sense, the sequence of errors of sampling discretization of the square norm of this class. Then we use this result for establishing new error bounds for sampling discretization of the square norm on classes of multivariate functions with mixed smoothness. }
\end{abstract}

\section{Introduction}

This paper is devoted to a study of discretization of the $L_2$ norm of continuous functions.
It is a follow up to the paper \cite{VT171}. 
Recently, a systematic study of the problem of discretization of the $L_q$ norms of elements of finite dimensional subspaces has begun. The reader can find a discussion of these results in 
the surveys \cite{DPTT} and \cite{KKLT}.
 There are different ways to discretize: use coefficients from an expansion with respect to a basis or, more generally, use 
linear functionals. We discuss here the way which uses function values at a fixed finite set of points. We call this way of discretization {\it sampling discretization}. 
 An important ingredient of this paper is that (alike the paper \cite{VT171}) we consider infinitely dimensional classes of functions and instead of the relative error setting, which was considered in  \cite{VT158} and \cite{VT159},  we consider the absolute error setting. In this section we formulate main results
 of the paper. We present a discussion of some new results in Section \ref{Disc}. 

{\bf Sampling discretization with absolute error.} Let $W\subset L_q(\Omega,\mu)$, $1\le q<\infty$, be a class of continuous on $\Omega$ functions. We are interested in estimating 
the following optimal errors of discretization of the $L_q$ norm of functions from $W$
$$
er_m(W,L_q):= \inf_{\xi^1,\dots,\xi^m} \sup_{f\in W} \left|\|f\|_q^q - \frac{1}{m}\sum_{j=1}^m |f(\xi^j)|^q\right|,
$$
$$
er_m^o(W,L_q):= \inf_{\xi^1,\dots,\xi^m;\la_1,\dots,\la_m} \sup_{f\in W} \left|\|f\|_q^q - \sum_{j=1}^m \la_j |f(\xi^j)|^q\right|.
$$

In this paper we only discuss in detail the case $q=2$. For this reason, in case $q=2$ we drop 
$L_q$ from notation: $er_m(W):=er_m(W,L_2)$ and  $er^o_m(W):=er^o_m(W,L_2)$. 

We demonstrate how known results from two areas of research -- supervised learning theory and numerical integration -- can be used in sampling discretization. 
We now formulate some typical results obtained in the paper. In Section \ref{B}  we prove the following result (see Section \ref{B} for the definition of the entropy numbers).

\begin{Theorem}\label{IT1} Assume that a class of real functions $W$ is such that for all $f\in W$ we have $\|f\|_\infty \le M$ with some constant $M$. Also assume that the entropy numbers of $W$ in the uniform norm $L_\infty$ satisfy the condition
$$
  \e_n(W) \le n^{-r} (\log (n+1))^b,\qquad r\in (0,1/2),\quad b\ge 0,\quad n\in \N.
$$
Then
$$
er_m(W):= er_m(W,L_2)  \le C(M,r,b)m^{-r} (\log (m+1))^b,\quad m\in \N.
$$
\end{Theorem}

In the case $b=0$ Theorem \ref{IT1} was proved in \cite{VT171}.  Theorem \ref{IT1} is a rather general theorem, which connects the behavior of absolute errors of discretization with the rate of decay of the entropy numbers. This theorem is derived in Section \ref{B} from known results in supervised learning theory. It is well understood in learning theory (see, for example, \cite{VTbook}, Ch.4) that the entropy numbers of the class 
of priors (regression functions) is the right characteristic in studying the regression problem. 

For the reader's convenience we write $a_m\ll b_m$ instead of $a_m\le Cb_m$, where $C$    is  positive constant independent of $m$. In case $a_m\ll b_m$ and $b_m\ll a_m$ we write $a_m\asymp b_m$.

In Section \ref{C} we apply Theorem \ref{IT1} to classes $\bW^r_p$ and $\bH^r_p$ of multivariate functions with small smoothness (see Section \ref{C} for their definition). We prove there the following two upper bounds.

\begin{Theorem}\label{CT5} Let $d\ge 2$, $2<p\le\infty$, and $1/p<r<1/2$. Then
$$
er_m(\bW^r_p,L_2) \ll m^{-r} (\log m)^{(d-1)(1-r)+r}.
$$
\end{Theorem}

\begin{Theorem}\label{CT6} Let $d\ge 2$, $2<p\le\infty$, and $1/p<r<1/2$. Then
$$
er_m(\bH^r_p,L_2) \ll m^{-r} (\log m)^{d-1+r}.
$$
\end{Theorem}

In Section \ref{E} we discuss a connection between characteristics $er_m^o(W)$, $er_m(W)$ and errors of numerical integration of functions from the class $W$. In the paper \cite{VT171} 
it was established that quasi-algebra property of the class $W$ allows us to obtain an upper estimates for the $er_m^o(W)$ and $er_m(W)$ in terms of errors of numerical integration of functions from the class $W$ (see Proposition \ref{EP1} in Section \ref{E}). In Section \ref{F} we 
prove that the classes $\bH^r$ (see the definition in Section \ref{C}) have the quasi-algebra property and therefore we can apply the technique from \cite{VT171}. For example, we prove in 
Section \ref{F} the following theorem.

\begin{Theorem}\label{FT1} Let $1\le p\le \infty$ and $r>1/p$. Then
$$
  er_m^o(\bH^r_p,L_2) \asymp m^{-r} (\log m)^{d-1}.
$$
\end{Theorem}

We note that the upper bound in Theorem \ref{FT1} is better than the one in Theorem \ref{CT6}. 
However, Theorem \ref{FT1} gives a bound for the $er_m^o(\bH^r_p,L_2)$ while Theorem \ref{CT6} gives a bound for the $er_m(\bH^r_p,L_2)$. We do not know if in Theorem \ref{FT1} 
the quantity $er_m^o(\bH^r_p,L_2)$ can be replaced by $er_m(\bH^r_p,L_2)$ in the case $d\ge 3$. We only know that this can be done in a special case of $p=\infty$ and $0<r<1$ (see Theorem \ref{FT2} below). For the case $d=2$ see Theorem \ref{FT3} below.

\section{Proof of Theorem \ref{IT1}}
\label{B}

In our further discussion we are interested in discretization of the $L_2$ norm of real functions from a given function class $W$. It turns out that 
this problem is closely related to some problems from supervised learning theory. We give a brief introduction to these problems. This is a vast area of research with a wide range of different settings. In this subsection we only discuss a development of a setting from \cite{CS} (see \cite{VTbook}, Ch.4, for detailed discussion). 

Let $X\subset\R^d$, $Y\subset \R$ be Borel sets, $\rho$ be a Borel probability measure on a Borel set $Z\subset X\times Y$.  For $f:X\to Y$ define {\it the error} 
$$
\cE(f)  :=\int_Z(f(\bx)-y)^2 d\rho.
$$
 
Let   $\rho_X$  be the marginal probability  measure of $\rho$ on $X$, i.e.,  $\rho_X (S) = \rho(S\times Y)  $ for Borel sets $S\subset X$. Define
$$
f_\rho(\bx) := \bbE(y|\bx)
$$
to be a conditional expectation of $y$.
The function $f_\rho$ is known in statistics as the {\it regression function} of $\rho$. It is clear that if $f_\rho\in L_2(\rho_X)$ then it minimizes the error $\cE(f)$ over all $f\in L_2(\rho_X)$ such that $\cE(f_\rho)\le \cE(f)$. Thus, in the sense of error $\cE(\cdot)$ the regression function $f_\rho$ is the best to describe the relation between inputs $\bx\in X$ and outputs $y\in Y$. The goal is to find an estimator $f_\bz$, on the base of given data $\bz:=((\bx^1,y_1),\dots,(\bx^m,y_m))$ that approximates $f_\rho$  well with high probability. We assume that $(\bx^i,y_i)$, $i=1,\dots,m$ are independent and distributed according to $\rho$.   We measure the error between $f_\bz$ and $f_\rho$ in the $L_2(\rho_X)$ norm. 

We note that a standard setting in the distribution-free theory of regression (see \cite{GKKW}) involves the expectation $\bbE(\|f_\rho- f_\bz\|_{L_2(\rho_X)}^2)$ as a measure of quality of an estimator. An important new feature of  the setting in learning theory formulated in \cite{CS} (see \cite{VTbook} for detailed discussion) is the following.  They propose to 
study systematically the probability distribution function
$$
\rho^m\{\bz:\|f_\rho-f_\bz\|_{L_2(\rho_X)}\ge\eta\}
$$
instead of the expectation. 

For a compact subset $\Theta$ of a Banach space $B$ we define the entropy numbers as follows
$$
\e_n(\Theta,B) := \inf\{\e: \exists f_1,\dots,f_{2^n}\in \Theta: \Theta\subset \cup_{j=1}^{2^n} (f_j+\e U(B))\}
$$
where $U(B)$ is the unit ball of a Banach space $B$. 
  
In   this subsection we always assume that the measure $\rho$ is concentrated on a bounded with respect to $y$ set, i.e. the set $Z$ satisfies the condition $Z\subset X\times [-M,M]$ (or a little weaker $|y|\le M$ a.e. with respect to $\rho_X$, i.e the $\rho_X$-measure of those $\bx$, for which there exists a $y$ such that $(\bx,y)\in Z$ and $|y|>M$ is equal to zero) with some fixed $M$. Then it is clear that for $f_\rho$ we have $|f_\rho(\bx)|\le M$ for all $\bx$ (for almost all $\bx$). Therefore, it is natural to assume that a class $\Theta$ of priors where $f_\rho$ belongs is embedded into the $\C(X)$-ball ($L_\infty$-ball) of radius $M$.  

We define the {\it empirical error} of $f$ as
$$
\E_\bz(f):= \frac{1}{m}\sum_{i=1}^m(f(\bx^i)-y_i)^2.
$$
Let $f\in L_2(\rho_X)$. The {\it defect function} of $f$ is
$$
L_\bz(f) := L_{\bz,\rho}(f) := \E(f)-\E_\bz(f);\quad \bz=(z_1,\dots,z_m),\quad z_i=(\bx^i,y_i).
$$
We are interested in estimating $L_\bz(f)$ for functions $f$ coming from a given class $W$. 
We begin with the case $B$ being 
 $\C(X)$, the space of functions continuous on a compact subset $X$ of $\R^d$ with the norm
$$
\|f\|_\infty:= \sup_{\bx\in X}|f(\bx)|.
$$
  We use the abbreviated notation
$$
  \e_n(W):= \e_n(W,\C).
$$

Settings for the supervised learning problem and the discretization problem are different. 
In the supervised learning problem we are given a sample $\bz$ and we want to approximately recover the regression function $f_\rho$. It is important that we do not know 
$\rho$. We only assume that we know that $f_\rho\in \Theta$. In the discretization of the $L_q$, $1\le q<\infty$, norm we assume that $f\in W$ and the probability measure $\mu$ is known. We want to find a discretization set $\xi=\{\bx^j\}_{j=1}^m$, which is good for the whole class $W$. However, the technique, based on the defect function, for solving the supervised learning problem can be used for solving the discretization problem. We now explain this in detail. Let us consider a given function class $W$ of real functions, defined on $X=\Omega$.
Suppose that the probability measure $\rho$ is such that $\rho_X=\mu$ and for all $\bx \in X$ 
we have $y=0$. In other words, we assume that $Y=\{0\}$. Then for the defect function we have
$$
L_\bz(f) =\int_X f^2d\mu -\frac{1}{m}\sum_{j=1}^m f(\bx^j)^2 =: L^2_{(\bx^1,\dots,\bx^m)}(f)
$$
and 
$$
\rho^m\{\bz:\sup_{f\in W} |L_\bz(f)|\ge \eta\} = \mu^m\{\bw:\sup_{f\in W} |L_\bw^2(f)|\ge \eta\}.
$$
Moreover, condition (\ref{B6}) (see below) is satisfied with $M$ such that for all $f\in W$ we have $\|f\|_\infty \le M$. The above argument shows that we can derive results on discretization of the $L_2$ norm directly from the corresponding results from learning theory. We assume that $W$ satisfies the following condition:
\be\label{B8}
f\in W\quad \Rightarrow \quad \|f\|_\infty \le M.
\ee

Our proof is based on the following known result. The following Theorem \ref{BT3} and Corollary   \ref{BC2} are from \cite{VT98} (see also \cite{VTbook}, section 4.3.3, p.213).   We assume that $\rho$ and $W$ satisfy the following condition. 
\begin{equation}\label{B6}
\text{For all}\quad f\in W,\quad f:X\to Y\quad \text{and any}\, (\bx,y)\in Z,\quad |f(\bx)-y| \le M.  
\end{equation}
 
\begin{Theorem}\label{BT3} Assume that $\rho$, $W$ satisfy (\ref{B6}) and $W$ is such that
$$
\sum_{n=1}^\infty n^{-1/2}\e_n(W) =\infty.
$$
For $\eta>0$ define $J:=J(\eta/M)$ as the minimal $j$ satisfying $\e_{2^j}\le \eta/(8M)$ and
$$
S_J:= \sum_{j=1}^J2^{(j+1)/2}\e_{2^{j-1}}.
$$
Then for $m$, $\eta$ satisfying $m(\eta/S_J)^2 \ge 480M^2$ we have
$$
\rho^m\{\bz:\sup_{f\in W} |L_\bz(f)|\ge \eta\} \le C(M,\e(W))\exp(-c(M)m(\eta/S_J)^2).
$$
\end{Theorem}
 
\begin{Corollary}\label{BC2} Assume $\rho$, $W$ satisfy (\ref{B6}) and 
$\e_n(W)\le Dn^{-r}$, $r\in (0,1/2)$.
Then for $m$, $\eta$ satisfying $m \eta^{1/r} \ge C_1(M,D,r)$ we have
$$
\rho^m\{\bz:\sup_{f\in W} |L_\bz(f)|\ge \eta\} \le C(M,D,r)\exp(-c(M,D,r)m\eta^{1/r} ).
$$
\end{Corollary}

We now derive from Theorem \ref{BT3} the following version of Corollary \ref{BC2}. 

\begin{Corollary}\label{BC3} Assume $\rho$, $W$ satisfy (\ref{B6}) and 
$$
\e_n(W)\le n^{-r}(\log (n+1))^b,\qquad r\in (0,1/2),\quad b\ge 0.
$$
Then there are three positive constants $C_i=C_i(M,r,b)$, $i=1,2$, $c=c(M,r,b)$ such that  for $m$, $\eta\le M$, satisfying $m \eta^{1/r} (\log (8M/\eta))^{-b/r} \ge C_1$ we have
\be\label{Co2}
\rho^m\{\bz:\sup_{f\in W} |L_\bz(f)|\ge \eta\} \le C_2 \exp(- cm\eta^{1/r} (\log (8M/\eta))^{-b/r} ).
\ee
\end{Corollary}
\begin{proof} We begin with a simple technical lemma.
\begin{Lemma}\label{BL1} Let $r>0$, $b\ge 0$, and $A\ge 2$. Then for $n\in \N$ the inequality 
\be\label{B1}
2^{rn}n^{-b} \le A
\ee
implies inequalities
$$
n\le C_1(r,b)\log A,\qquad 2^n \le C_2(r,b)A^{1/r}(\log A)^{b/r}
$$
with some positive constants $C_i(r,b)$, $i=1,2$.
\end{Lemma}
\begin{proof} Set
$$
C_0(r,b) := \max_{n\in \N} n^b2^{-rn/2}.
$$
Then (\ref{B1}) implies
$$
2^{rn/2}  \le AC_0(r,b)\quad \text{and} \quad n\le C_1(r,b)\log A.
$$
This bound and (\ref{B1}) imply the required bound on $2^n$.
\end{proof}

We continue the proof of Corollary \ref{BC3}. Take $\eta \le M$. From the definition of $J$ in Theorem \ref{BT3} we obtain
$$
\e_{2^J}\le \eta/(8M)\quad \text{and}\quad \e_{2^{J-1}}> \eta/(8M).
$$
Therefore,
$$
2^{-r(J-1)}(J-1)^b > \eta/(8M) \quad \text{and}\quad 2^{r(J-1)}(J-1)^{-b} < 8M/\eta.
$$
By Lemma \ref{BL1} with $A=8M/\eta$ we obtain
\be\label{B2}
J-1\le C_1(r,b)\log (8M/\eta),\qquad 2^{J-1} \le C_2(r,b)(8M/\eta)^{1/r}(\log (8M/\eta))^{b/r}.
\ee
For $S_J$ we obtain the upper bound
$$
S_J= \sum_{j=1}^J2^{(j+1)/2}\e_{2^{j-1}} \le  \sum_{j=1}^J2^{(j+1)/2}2^{-r(j-1)}j^b   
$$ 
\be\label{B3}
\le C_3(r,b) 2^{(1/2-r)J}J^b \le C_4(r,b) (8M/\eta)^{(1/2-r)/r} (\log (8M/\eta))^{b/(2r)}.
\ee
For $m\ge 2$ set $\eta_m := am^{-r} (\log m)^b$. Bound (\ref{B3}) implies that there is a large enough $a=a(M,r,b)$ such that for $\eta\ge \eta_m$ the inequalities $m(\eta/S_J)^2 \ge 480M^2$ and $(\eta/S_J)^2 \ge C_5(M,r,b)\eta^{1/r} (\log (8M/\eta))^{-b/r}$
are satisfied. 

\end{proof}

Clearly, it is sufficient to prove Theorem \ref{IT1} for $m\ge m_0 = m_0(M,r,b)$.  The statement of Theorem \ref{IT1} follows from Corollary \ref{BC3}. Indeed, setting $\eta_m := am^{-r} (\log m)^b$ and choosing $a=a(M,r,b)$ large enough we satisfy the condition  
$m \eta_m^{1/r} (\log (8M/\eta_m))^{-b/r} \ge C_1$ of the Corollary \ref{BC3} and make the right hand side of (\ref{Co2}) less than 1. Choosing $m_0$ such that $\eta_m\le M$ for $m\ge m_0$, 
we apply Corollary \ref{BC3} and complete the proof. 

\section{Discretization for classes with small mixed smoothness}
\label{C}

We now proceed to applications of Theorem \ref{IT1} to classes of functions with mixed smoothness. 
 Classically, Sobolev classes of mixed smoothess were defined via $L_p$-boundedness of mixed weak derivatives in contrast to weak derivatives of order $r$ for the usual Sobolev classes. If $r \in \N_0$ and $1 \leq p \leq \infty$ we define $\bW^r_p$ as the set of all $L_p(\T^d)$ functions such that 
$$
	\sum\limits_{e \subset [d]} \Big\|\prod\limits_{i \in e}\Big(\frac{\partial^r}{\partial x_i^r}\Big)f\Big\|_p \leq 1\,.
$$
If $r$ is not an integer we use the following characterization of the classes $\bW^r_{p}$ in the standard in approximation theory and for us convenient way. See also \cite{DTU}, Ch.3, for further equivalent characterizations and connections to, e.g., Triebel-Lizorkin spaces. For $r > 0$   the functions
$$
F_r(x):=1+2\sum_{k=1}^{\infty}k^{-r}\cos(kx - r\pi/2)
$$
are called {\it Bernoulli kernels}.
 Let  
$$
F_r (\bx) := \prod_{j=1}^d F_r (x_j)
$$
be the multivariate analog of the Bernoulli kernel. 
We denote by $\widetilde{\bW}_{p}^r$ the class of functions
$f(\bx)$ representable in the form
$$
f(\bx) =\varphi(\bx)\ast F_r (\bx) :=
(2\pi)^{-d}\int_{\T^d}\varphi(\by)F_r (\bx-\by)d\by,
$$
where $\varphi\in L_p$ and $\|\varphi\|_p\le 1$.  It is well known that in case $1<p<\infty$ and $r\in \N$ these two characterizations ($\bW^r_p$ and $\widetilde{\bW}^r_p$) are equivalent in the sense that there are constants $A(r,p)>a(r,p)>0$ such that 
$$a\cdot \bW^r_p \subset \widetilde{\bW}^r_p \subset A\cdot \bW^r_p\,.$$ Therefore we just use one notation, namely $\bW^r_p$ in the sequel.     

We formulate some known results for the Kolmogorov widths: For a compact set $W \subset X$ of a Banach space $X$ define
$$
d_m(W,X) := \inf_{\{u_i\}_{i=1}^m\subset X}
\sup_{f\in W}
\inf_{c_i} \left \| f - \sum_{i=1}^{m}
c_i u_i \right\|_X,\quad m=1,2,\dots
$$
and
$$
d_0(W,X):= \sup_{f\in W}\|f\|_X.
$$
The following bound for the $ d_m(\bW^r_p,L_\infty)$ was obtained in \cite{VT184}.

\begin{Theorem}\label{CT1} Let $d\ge 2$, $2<p\le\infty$, and $1/p<r<1/2$. Then
$$
d_m(\bW^r_p,L_\infty) \ll m^{-r} (\log m)^{(d-1)(1-r)+r}.
$$
\end{Theorem}

 We now turn our discussion to the classes $\bH^r_p$.
Let  $\btt =(t_1,\dots,t_d )$ and $\Delta_{\btt}^l f(\bx)$
be the mixed $l$-th difference with
step $t_j$ in the variable $x_j$, that is
$$
\Delta_{\btt}^l f(\bx) :=\Delta_{t_d,d}^l\dots\Delta_{t_1,1}^l
f(x_1,\dots ,x_d ) .
$$
Let $e$ be a subset of natural numbers in $[1,d ]$. We denote
$$
\Delta_{\btt}^l (e) =\prod_{j\in e}\Delta_{t_j,j}^l,\qquad
\Delta_{\btt}^l (\varnothing) = I .
$$
We define the class $\bH_{p,l}^r B$, $l > r$, as the set of
$f\in L_p$ such that for any $e$
\be\label{C1}
\bigl\|\Delta_{\btt}^l(e)f(\bx)\bigr\|_p\le B
\prod_{j\in e} |t_j |^r .
\ee
In the case $B=1$ we omit it. It is known (see, for instance, \cite{VTbookMA}, p.137) that the classes $\bH^r_{p,l}$ with different $l$ are equivalent. So, for convenience we fix one $l= [r]+1$ and omit $l$ from the notation. 
 The following bound for the $ d_m(\bH^r_p,L_\infty)$ was obtained in \cite{VT184}.

\begin{Theorem}\label{CT2} Let $d\ge 2$, $2<p\le\infty$, and $1/p<r<1/2$. Then
$$
d_m(\bH^r_p,L_\infty) \ll m^{-r} (\log m)^{d-1+r}.
$$
\end{Theorem}

There are several general results, which give 
upper estimates
on the entropy numbers
$\e_k(F,X)$   in terms of the Kolmogorov widths $d_n(F,X)$.  Carl's (see \cite{Ca} and \cite{VTbook}, p.169, Theorem 3.23)
inequality states: For any $r>0$ we have
\begin{equation}\label{Carl}
\max_{1\le k \le n} k^r \e_k(F,X) \le C(r) \max _{1\le m \le n} m^r d_{m-1}(F,X).
\end{equation}
Inequality (\ref{Carl}) and Theorems \ref{CT1} and \ref{CT2} imply (see also \cite{VT185}) 
the following two theorems. 

\begin{Theorem}\label{CT3} Let $d\ge 2$, $2<p\le\infty$, and $1/p<r<1/2$. Then
$$
\e_k(\bW^r_p,L_\infty) \ll k^{-r} (\log (k+1))^{(d-1)(1-r)+r}.
$$
\end{Theorem}

\begin{Theorem}\label{CT4} Let $d\ge 2$, $2<p\le\infty$, and $1/p<r<1/2$. Then
$$
\e_k(\bH^r_p,L_\infty) \ll k^{-r} (\log (k+1))^{d-1+r}.
$$
\end{Theorem}

We now combine Theorems \ref{CT3} and \ref{CT4} with Theorem \ref{IT1} in order to obtain
the corresponding discretization results. It is well known and easy to check that for $r>1/p$ there exists a constant $C(d,r,p)$ such that for all $f\in \bW^r_p$ and $f\in \bH^r_p$ we have
$\|f\|_\infty \le C(d,r,p)$. Therefore, Theorem \ref{CT3} and Theorem \ref{IT1} give Theorem \ref{CT5} and Theorem \ref{CT4} and Theorem \ref{IT1} give Theorem \ref{CT6}.

\section{A connection to numerical integration}
\label{E}

For a compact subset $W\subset \C(\Omega)$ define the best error of numerical integration with $m$ knots as follows
$$
\kappa_m(W) := \inf_{\xi^1,\dots,\xi^m;\lambda_1,\dots,\lambda_m} \sup_{f\in W}\left|\int_\Omega fd\mu - \sum_{j=1}^m \lambda_j f(\xi^j)\right|.
$$
For a set of points $\xi=\{\xi^j\}_{j=1}^m$ and a set of weights $\{\la_j\}_{j=1}^m$ define the cubature formula
$$
\La_m(f,\xi):= \sum_{j=1}^m \lambda_j f(\xi^j).
$$
Also, define the best error of numerical integration by Quasi-Monte Carlo methods with $m$ knots as follows
$$
\kappa_m^Q(W) := \inf_{\xi^1,\dots,\xi^m} \sup_{f\in W}\left|\int_\Omega fd\mu - \frac{1}{m} \sum_{j=1}^m  f(\xi^j)\right|.
$$
Obviously, $\kappa_m^Q(W)\ge \kappa_m(W)$. 

We begin with a very simple general observation on a connection between norm discretization and numerical integration (see \cite{VT171}). 

{\bf Quasi-algebra property.} We say that a function class $W$ has the quasi-algebra property (with a parameter $a$) if  there exists a constant $a$ such that for any $f,g\in W$ we have $fg/a \in W$.

The above property was introduced and studied in detail by H. Triebel. He introduced this property under the name {\it multiplication algebra}. Normally, the term {\it algebra} refers to 
the corresponding property with parameter $a=1$. To avoid any possible confusions we call 
it {\it quasi-algebra}. We refer the reader to the very resent book of Triebel \cite{Tr}, which contains results on the multiplication algebra (quasi-algebra) property for a broad range of
function spaces. 

We now formulate a simple statement from \cite{VT171}, which gives a connection between numerical integration and discretization of the $L_2$ norm. 

\begin{Proposition}\label{EP1} Suppose that a function class $W$ has the quasi-algebra property with a parameter $a$  and for any $f\in W$ we have for the complex conjugate function $\bar f\in W$. 
  Then for a cubature formula $\Lambda_m(\cdot,\xi)$ we have: For any $f\in W$
$$
|\|f\|_2^2 -\Lambda_m(|f|^2,\xi)| \le a\sup_{g\in W} \left|\int_\Omega gd\mu - \Lambda_m(g,\xi)\right|.
$$
\end{Proposition}

The lower bound for discretization in terms of errors of numerical integration holds under very mild conditions on the class $W$ (see \cite{VT171}). We now proceed to the case of discretization of the $L_2$ norm. In this case it is convenient for us to consider real functions. Assume that a class of real functions $W\subset \C(\Omega)$ has the following extra property. 

{\bf Property A.} For any $f\in W$ we have $f^+ := (f+1)/2 \in W$ and $f^- := (f-1)/2 \in W$.

In particular, this property is satisfied if $W$ is a convex set containing functions $1$ and $-1$. The following result is from \cite{VT171}.

\begin{Theorem}\label{ET1} Suppose $W\subset \C(\Omega)$ has Property A. Then for any $m\in \N$ we have
$$
er_m^o(W,L_2) \ge \frac{1}{2}\kappa_m(W).
$$
\end{Theorem}

We note that the proof of Theorem \ref{ET1} from \cite{VT171} gives the following version of 
Theorem \ref{ET1}.

\begin{Theorem}\label{ET1Q} Suppose $W\subset \C(\Omega)$ has Property A. Then for any $m\in \N$ we have
$$
er_m(W,L_2) \ge \frac{1}{2}\kappa_m^Q(W).
$$
\end{Theorem}

Proposition \ref{EP1} and Theorems \ref{ET1} and \ref{ET1Q} imply the following general statement.

\begin{Theorem}\label{ET2} Suppose that a function class $W\subset \C(\Omega)$ of real functions has the quasi-algebra property with a parameter $a$ and   has Property A. Then for any $m\in \N$ we have
$$
 \frac{1}{2}\kappa_m(W) \le er_m^o(W,L_2) \le a \kappa_m(W)
$$
and
$$
 \frac{1}{2}\kappa_m^Q(W) \le er_m(W,L_2) \le a \kappa_m^Q(W).
$$

\end{Theorem}

Theorems \ref{IT1} and \ref{ET1Q} imply the following inequalities, which provide an upper bound for the errors of numerical integration in terms of the entropy numbers. 

\begin{Theorem}\label{ET3} Assume that a class of real functions   $W\subset \C(\Omega)$ has Property A and is such that for all $f\in W$ we have $\|f\|_\infty \le M$ with some constant $M$. Also assume that the entropy numbers of $W$ in the uniform norm $L_\infty$ satisfy the condition
$$
  \e_n(W) \le n^{-r} (\log (n+1))^b,\qquad r\in (0,1/2),\quad b\ge 0.
$$
Then
$$
\kappa_m(W)   \le \kappa_m^Q(W)   \le C(M,r,b)m^{-r} (\log (m+1))^b.
$$
\end{Theorem}

The following inequality was proved in \cite{No} (see also \cite{NoLN})
\be\label{E2}
\kappa_m(W) \le 2d_m(W,L_\infty),
\ee
where $d_m(W,L_\infty)$ is the Kolmogorov width of $W$ in the uniform norm $L_\infty$.

\section{Further results on discretization}
\label{F}

In this section we present some results for the $\bH$-classes defined in Section \ref{C}. It will be convenient for us to use a known representation theorem for the $\bH^r_p$ classes. 
 We need some classical trigonometric polynomials for our further argument (see \cite{Z} and \cite{VTbookMA}). We begin with the univariate case. 
 The Dirichlet kernel of order $j$:
$$
\mathcal D_j (x):= \sum_{|k|\le j}e^{ikx} = e^{-ijx} (e^{i(2j+1)x} - 1)
(e^{ix} - 1)^{-1} 
$$
$$
=\bigl(\sin (j + 1/2)x\bigr)\bigm/\sin (x/2)
$$
   is an even trigonometric polynomial.  The Fej\'er kernel of order $j - 1$:
$$
\mathcal K_{j} (x) := j^{-1}\sum_{k=0}^{j-1} \mathcal D_k (x) =
\sum_{|k|\le n} \bigl(1 - |k|/j\bigr) e^{ikx} 
$$
$$
=\bigl(\sin (jx/2)\bigr)^2\bigm /\bigl(j (\sin (x/2)\bigr)^2\bigr).
$$
The Fej\'er kernel is an even nonnegative trigonometric
polynomial in $\Tr(j-1)$.  It satisfies the obvious relations
\be\label{FKm}
\| \mathcal K_{j} \|_1 = 1, \qquad \| \mathcal K_{j} \|_{\infty} = j.
\ee
The de la Vall\'ee Poussin kernel
\be\label{A2}
\mathcal V_{j} (x) := j^{-1}\sum_{l=j}^{2j-1} \mathcal D_l (x)= 2\cK_{2j}(x)-\cK_j(x) 
\ee
is an even trigonometric
polynomial of order $2j - 1$.

Consider the following special univariate trigonometric polynomials. Let $s$ be a nonnegative integer. Define
$$
\cA_0(x) := 1,\quad \cA_1(x):= \cV_1(x)-1,\quad \cA_s(x):= \cV_{2^{s-1}}(x) - \cV_{2^{s-2}}(x), \quad s\ge 2,
$$
where $\cV_j$ are the de la Vall\'ee Poussin kernels defined above. Then $\cA_s(x)$ is a trigonometric polynomial of degree $2^s-1$. For $s\ge 2$ we have $\hat \cA_s(k) =0$ for 
$|k|\le 2^{s-2}$. Thus, for $s\ge 2$ the $\hat \cA_s(x)$ has nonzero Fourier coefficients for $2^{s-2} < |k|< 2^s$. 

In the multivariate case 
$\bx=(x_1,\dots,x_d)$ and $\bs=(s_1,\dots,s_d)\in \Z_+^d$ define
$$
\cA_\bs(\bx):= \cA_{s_1}(x_1)\cdots \cA_{s_d}(x_d).
$$
For $f\in L_1(\T^d)$ denote
$$
A_\bs(f)(\bx) := (f\ast \cA_s)(\bx) := (2\pi)^{-d} \int_{\T^d} f(\by)\cA_s(\bx-\by)d\by,
$$
where $\T^d := [0,2\pi)^d$.

It is known (see \cite{VTbookMA}, p.137) that an equivalent definition of the class $\bH^r_pB$ is as follows:
$$
\bH^r_pB :=\{f\in L_p(\T^d)\,:\, \|A_\bs\|_p \le B2^{-r\|\bs\|_1},\quad \bs \in \N_0^d\}.
$$

\begin{Proposition}\label{FP1} Let $r>1/p$. There exists $B=B(d,r,p)$ such that for any two functions $f,g \in \bH^r_p$ we have $fg \in \bH^r_pB$.
\end{Proposition}
\begin{proof} We begin with the univariate case $d=1$. Assumption $r>1/p$ guarantees that for $f,g \in \bH^r_p$ we have absolutely convergent representations 
$$
f(x) = \sum_{s=0}^\infty A_s(f)(x),\qquad g(x) = \sum_{s=0}^\infty A_s(g)(x).
$$
Therefore,
$$
f(x)g(x) = \sum_{u=0}^\infty \sum_{v=0}^\infty A_u(f)(x)A_v(g)(x).
$$
Consider $A_s(fg)$. Let $u\ge v$. Then, it is clear that 
$$
A_s(A_u(f)A_v(g)) =0 \quad \text{if}\quad 2^u + 2^v \le 2^{s-2}.
$$
We now split the summation over $u,v$ into two regions $D_1: \{u,v\in N_0: u\ge v\}$ and
$D_2: \{u,v\in N_0: u< v\}$. Then we have
$$
\Sigma_1:=\left\|A_s\left(\sum_{u,v\in D_1} A_u(f)A_v(g)\right)\right\|_p =\left\|\left(\sum_{u=\max(s-3,0)}^\infty\sum_{0\le v\le u} A_s(A_u(f)A_v(g))\right)\right\|_p
$$
$$
\le 6 \sum_{u=\max(s-3,0)}^\infty\sum_{0\le v\le u} \|A_u(f)A_v(g)\|_p \le 6 \sum_{u=\max(s-3,0)}^\infty\sum_{0\le v\le u} \|A_u(f)\|_p\|A_v(g)\|_\infty.
$$
Using the definition of the class $\bH^r_p$ and the Nikol'skii inequality we obtain 
\be\label{F3}
\|A_u(f)\|_p \le 2^{-ru},\qquad \|A_v(g)\|_\infty \le C2^{v/p}\|A_v(g)\|_p \le C2^{-(r-1/p)v}.
\ee
Thus, we obtain
$$
\Sigma_1 \le C'2^{-rs}.
$$
In the same way we bound the sum $\Sigma_2$ and complete the proof in the univariate case. 

In the multivariate case we do the same coordinate wise. If for some $j$ we have 
$2^{u_j} + 2^{v_j} \le 2^{s_j-2}$ then $A_\bs(A_\bu(f)A_\bv(g)) =0$. We now split the summation over $2^d$ regions, where either $u_j\ge v_j$ or $u_j<v_j$, $j=1,\dots,d$. Let $e$ be a subset 
of $\{1,2,\dots,d\}$ and $e^c:=\{1,2,\dots,d\}\setminus e$. Denote $\bp_e =(p_1,\dots,p_d)$ with $p_j=p$ for $j\in e$ and $p_j=\infty$ 
otherwise. Let 
$$
D_e:= \{\bu,\bv \,:\, u_j\ge v_j, \, j\in e, \, u_j<v_j \, \text{otherwise}\}.
$$
Then we bound $\|A_\bs(A_\bu(f)A_\bv(g))\|_p$ as follows
$$
\|A_\bs(A_\bu(f)A_\bv(g))\|_p \le 6^d\| A_\bu(f)A_\bv(g)\|_p \le 6^d \|A_\bu(f)\|_{\bp_e}\|A_\bv(g)\|_{\bp_{e^c}}.
$$
We use the Nikol'skii vector norm inequalities (see \cite{VTbookMA}, p.90) for each $\|A_\bu(f)\|_{\bp_e}$ and $\|A_\bv(g)\|_{\bp_{e^c}}$. 
Arguing as above we obtain that for all $\bs\in\N_0^d$ we have
$$
\|A_\bs(fg)\|_p \le C(d,r,p)2^{-r\|\bs\|_1},
$$
which completes the proof. 

\end{proof}

We now demonstrate how known results on the errors of numerical integration provide optimal rates of decay of the quantities $er_m^o(\bH^r_p,L_2)$ and $er_m(\bH^r_p,L_2)$. We will use 
Theorems \ref{ET1}, \ref{ET2} and Proposition \ref{EP1}. Our argument follows the one from
\cite{VT171}, where the corresponding results were obtained for the $\bW^r_2$ and $\bE^r$ classes. The following relation is known: Let $1\le p\le \infty$ and $r>1/p$. Then
\be\label{F4}
\kappa_m(\bH^r_p) \asymp m^{-r} (\log m)^{d-1}.
\ee
The lower bounds in (\ref{F4}) were obtained in \cite{Bakh4} (see also the book \cite{VTbookMA}, p.268, Theorem 6.4.8). The upper bounds in (\ref{F4}) were obtained in \cite{Du1} (for $r>1$ see also the book \cite{VTbookMA}, p.299, Theorem 6.7.13).

Bounds (\ref{F4}) Theorem \ref{ET2} and Proposition \ref{FP1} give Theorem \ref{FT1} from 
Introduction.

We now give one result for the $er_m(\bH^r_\infty,L_2)$. We recall the concept of the Korobov cubature formulas. Let $m\in\N$, $\mathbf a := (a_1,\dots,a_d)$, $a_1,\dots,a_d\in\Z$.
We consider the cubature formulas
$$
P_m (f,\mathbf a):= m^{-1}\sum_{\mu=1}^{m}f\left (2\pi\left \{\frac{\mu a_1}
{m}\right\},\dots,2\pi\left \{\frac{\mu a_d}{m}\right\}\right),
$$
which are called the {\it Korobov cubature formulas}. For a function class $W$ denote
$$
P_m(W,\ba) := \sup_{f\in W} |P_m(f,\ba)-\hat f(\mathbf 0)|.
$$
The following result is known (see \cite{VTbookMA}, p.288, Theorem 6.6.5): Let $0 < r < 1$. There is a vector $\mathbf a$
such that
\be\label{F5}
P_m (\bH_{\infty}^r ,\mathbf a)\ll m^{-r} (\log m)^{d-1}.
\ee

Bounds (\ref{F4}), (\ref{F5}), Theorem \ref{ET2} and Proposition \ref{FP1} give the following result.

\begin{Theorem}\label{FT2} Let $0 < r < 1$. Then
$$
  er_m(\bH^r_\infty,L_2) \asymp m^{-r} (\log m)^{d-1}.
$$
\end{Theorem}

In the case $d=2$ the results are complete. They are based on the Fibonacci cubature formulas, which we introduce momentarily. Let $\{b_n\}_{n=0}^{\infty}$, $b_0=b_1 =1$, $b_n = b_{n-1}+b_{n-2}$,
$n\ge 2$, -- be the Fibonacci numbers.  For the continuous functions of two
variables, which are $2\pi$-periodic in each variable, we define
cubature formulas
$$
\Phi_n(f) :=b_n^{-1}\sum_{\mu=1}^{b_n}f\bigl(2\pi\mu/b_n,
2\pi\{\mu b_{n-1} /b_n \}\bigr),
$$
 called the {\it Fibonacci cubature formulas}. The following result is known (see \cite{VTbookMA}, p.281, Theorem 6.5.8): Let $1\le p\le \infty$ and $r>1/p$. Then for $d=2$
 \be\label{F6}
 \sup_{f\in \bH^r_p}\left|(2\pi)^{-2}\int_{\T^2} f(\bx)d\bx - \Phi_n(f)\right| \asymp b_n^{-r}\log b_n.
 \ee
 
 Bounds (\ref{F4}), (\ref{F6}), Theorem \ref{ET2} and Proposition \ref{FP1} give the following result.

\begin{Theorem}\label{FT3} Let $1\le p\le \infty$ and $r>1/p$. Then for $d=2$
$$
  er_{b_n}(\bH^r_p,L_2) \asymp b_n^{-r} \log b_n.
$$
\end{Theorem}

\section{Discussion}
\label{Disc}

We begin our discussion with Theorem \ref{ET3}. This theorem under mild conditions on 
the function class $W$ guarantees that a certain decay of the entropy numbers:
\be\label{D1}
  \e_n(W) \le n^{-r} (\log (n+1))^b,\qquad r\in (0,1/2),\quad b\ge 0, \quad n\in \N
\ee
implies the same decay of the optimal errors of numerical integration
\be\label{D2}
\kappa_m(W)\le \kappa_m^Q(W)   \le C(M,r,b)m^{-r} (\log (m+1))^b,\quad m\in \N.
\ee
First of all, we point out that a single inequality (\ref{D1}) with $b=0$ (for simplicity) for an $n\in \N$ implies by the Hoeffding's inequality (see below) that for $m\asymp n^{1+2r}$ there exists a set of points $\{\xi^j\}_{j=1}^m$ such that we have for all $f\in W$
$$
\left|\int_\Omega fd\mu - \frac{1}{m} \sum_{j=1}^m  f(\xi^j)\right| \ll n^{-r} \asymp m^{-r/(1+2r)},
$$
 which is much weaker than (\ref{D2}).

Let $\Omega$ be a compact subset of $\R^d$ and $\mu$ be a probability measure on  $\Omega$.    Define $\Omega^m:= \Omega\times\cdots\times\Omega$ and $\mu^m:=\mu \times\cdots\times \mu$. For $\bx^j\in \Omega$ denote $\bz:= (\bx^1,\dots,\bx^m)\in \Omega^m$. Consider a real function $f\in \cC(\Omega)$.
Under the condition $\|f\|_\infty \le M$ the Hoeffding's inequality (see, for instance, \cite{VTbook}, p.197) gives
\be\label{D3}
\mu^m\left\{\bz: \left|\int_\Omega fd\mu - \frac{1}{m}\sum_{j=1}^m f(\bx^j)\right|\ge \eta\right\}\le 2\exp\left(-\frac{m\eta^2}{8M^2}\right).
\ee

Second, we recall that the following inequality was proved in \cite{No} (see also \cite{NoLN})
\be\label{D4}
\kappa_m(W) \le 2d_m(W,L_\infty)
\ee
where $d_m(W,L_\infty)$ is the Kolmogorov width of $W$ in the uniform norm $L_\infty$.
Inequality (\ref{D4}) is a very nice inequality. However, it gives an upper bound for the 
$\kappa_m(W)$, which is smaller than $\kappa_m^Q(W)$ in terms of the $d_m(W,L_\infty)$, which are larger (in the sense of Carl's inequality (see (\ref{Carl})) than the entropy numbers.

Third, we point out that the Carl's inequality (\ref{Carl})  and Theorem \ref{ET3} imply the following result.

\begin{Theorem}\label{DT1} Assume that a class of real functions   $W\subset \C(\Omega)$ has Property A and is such that for all $f\in W$ we have $\|f\|_\infty \le M$ with some constant $M$. Also assume that the Kolmogorov widths of $W$ in the uniform norm $L_\infty$ satisfy the condition
$$
  d_n(W,L_\infty) \le n^{-r} (\log (n+1))^b,\qquad r\in (0,1/2),\quad b\ge 0.
$$
Then
$$
\kappa_m(W)   \le \kappa_m^Q(W)   \le C(M,r,b)m^{-r} (\log (m+1))^b.
$$
\end{Theorem}

The sampling discretization errors $er_m(W,L_q)$ and $er_m^o(W,L_q)$ are new asymptotic characteristics of a function class $W$. It is natural to try to compare these characteristics with other classical asymptotic characteristics. Theorem \ref{IT1} addresses this issue. It is known that the sequence of entropy numbers is one of the smallest sequences of asymptotic characteristics of a class. For instance, by Carl's inequality (see \cite{Ca} and (\ref{Carl}) above) it is dominated, in a certain sense,
by the sequence of the Kolmogorov widths. Theorem \ref{IT1} shows that the sequence $\{\e_n(W)\}$ dominates, in a certain sense, the sequence $\{er_m(W)\}$. Clearly, alike the Carl's inequality, one tries to prove the corresponding relations in as general situation as possible. 
We derive Theorem \ref{IT1} from known results in learning theory. Our proof is a probabilistic 
one.  
We impose the restriction $r<1/2$ in Theorem \ref{IT1} because the probabilistic technique from the supervised learning theory, which was used in the proof of Theorem \ref{IT1} (see Section \ref{B}), has a natural limitation to $r\le 1/2$. As we pointed out in \cite{VT171} in case $b=0$, it would be interesting to understand if Theorem \ref{IT1} holds for $r\ge 1/2$. Similarly, it would be interesting to understand if it is possible to extend Theorems \ref{ET3} and \ref{DT1} to the case of $r\ge 1/2$. 
We point out that Theorem \ref{IT1} gives an upper bound for the quantity $er_m(W)$, which is 
a larger (in general) one than the optimized over weights quantity $er_m^o(W)$.

\end{document}